\documentclass[12pt]{article}
\usepackage{amsmath}
\usepackage{latexsym}
\usepackage{amssymb}
\newtheorem{thm}{Theorem}[section]
\newtheorem{la}[thm]{Lemma}
\newtheorem{Defn}[thm]{Definition}
\newtheorem{Remark}[thm]{Remark}
\newtheorem{prop}[thm]{Proposition}

\newtheorem{Number}[thm]{\!\!}

\newenvironment{numba}{\begin{Number}\rm}{\end{Number}}
\newenvironment{proof}{{\noindent\bf Proof.}}%
                  {\nopagebreak\hspace*{\fill}$\Box$\medskip\par}
\newcommand{\Punkt}{\nopagebreak\hspace*{\fill}$\Box$}
\newcommand{\wb}{\overline}
\newcommand{\ve}{\varepsilon}
\newcommand{\wt}{\widetilde}

\newcommand{\mto}{\mapsto}
\newcommand{\N}{{\mathbb N}}
\newcommand{\R}{{\mathbb R}}
\newcommand{\bS}{{\mathbb S}}
\newcommand{\cg}{{\mathfrak g}}
\newcommand{\cW}{{\mathcal W}}
\newcommand{\cL}{{\mathcal L}}
\newcommand{\cO}{{\mathcal O}}

\newcommand{\sub}{\subseteq}
\newcommand{\wh}{\widehat}
\DeclareMathOperator{\GL}{GL}
\DeclareMathOperator{\id}{id}
\newcommand{\sbull}{{\scriptscriptstyle \bullet}}
\DeclareMathOperator{\Diff}{Diff}
\DeclareMathOperator{\graph}{graph}
\DeclareMathOperator{\Lip}{Lip}
\DeclareMathOperator{\Evol}{Evol}
\DeclareMathOperator{\evol}{evol}
\DeclareMathOperator{\flt}{flat}

\begin{document}
\begin{center}
{\Large\bf Diffeomorphism groups of compact\\[2mm] convex~sets}\\[7mm]
{\bf Helge Gl\"{o}ckner and Karl-Hermann Neeb}\vspace{2mm}
\end{center}
\begin{abstract}
\hspace*{-6mm}For $K\sub\R^n$ a compact convex subset with non-empty interior,
let $\Diff_{\partial K}(K)$ be the group of all $C^\infty$-diffeomorphisms
of $K$ which fix $\partial K$ pointwise.
We show that $\Diff_{\partial K}(K)$ is a $C^0$-regular infinite-dimensional Lie group.
As a byproduct, we obtain results concerning solutions to
ordinary differential equations on compact convex sets.\vspace{2mm}
\end{abstract}
{\bf Classification:}
22E65 (primary); 34A12 (secondary)\\[2.3mm]
{\Large\bf Introduction and statement of the main results}\\[4mm]
Lie groups of smooth diffeomorphisms of compact manifolds (like the diffeomorphism group $\Diff(\bS^1)$
of the circle) are among the most prominent and important examples of infinite-dimensional
Lie groups (see, e.g., \cite{Ham}, \cite{KaM}, \cite{Les}, \cite{Mil}, \cite{Omo}; cf.\ \cite{PaS}).
A Lie group structure on $\Diff(K)$ even is available if $K$ is a compact manifold with boundary or corners
\cite{Mic}; this includes the case that $K\sub \R^n$ is a convex polyhedron.\footnote{Diffeomorphism groups of non-compact manifolds can also be treated; their Lie group structures are modelled only on Lie algebras
of compactly supported smooth vector fields. For Lie groups of real analytic diffeomorphisms of real analytic, compact manifolds
with or without boundary or corners,
cf.\ \cite{DaS}, \cite{Eyn}, and \cite{KaM}.}
In this article, we describe Lie groups of diffeomorphisms
of an arbitrary compact convex subset $K\sub \R^n$ with non-empty interior
(whose boundary $\partial K$ need not satisfy
any regularity assumptions).
To explain the result, let us call a map $\gamma\colon K\to\R^n$
\emph{smooth} if it is continuous, its restriction $\gamma|_{K^0}$ to
the interior of~$K$ is smooth, and all iterated directional derivatives
on~$K^0$ admit continuous extensions to all of~$K$
(see \ref{defnonop} for details).
We write $\Diff(K)$ for the group of all smooth diffeomorphisms of~$K$,
i.e., bijections $\phi\colon K\to K$ such that both $\phi$ and $\phi^{-1}$ are smooth
in the preceding sense.
We endow the space $C^\infty(K,\R^n)$ of all smooth $\R^n$-valued mappings on $K$
with the smooth compact-open topology (as recalled in \ref{defsmooCk}),
which makes it a Fr\'{e}chet space
(see \cite{GaN}, cf.\ \cite{AaS}).
Then also the closed vector subspace
\[
C^\infty_{\partial K}(K,\R^n):=\{\eta\in C^\infty(K,\R^n)\colon\eta|_{\partial K}=0\}
\]
of $C^\infty(K,\R^n)$
is a Fr\'{e}chet space.
Now
\[
\Diff_{\partial K}(K):=\{\phi\in \Diff(K)\colon \,(\forall x\in\partial K)\colon \phi(x)=x\}
\]
is a subgroup of~$\Diff(K)$.
We show that
\[
\Omega:=\{\phi-\id_K\colon \phi\in\Diff_{\partial K}(K)\}
\]
is an open $0$-neighbourhood in $C^\infty_{\partial K}(K,\R^n)$ (see Section~\ref{sec3}),
enabling us to consider $\Diff_{\partial K}(K)$ as a smooth manifold
modelled on $C^\infty_{\partial K}(K,\R^n)$ with
\[
\Diff_{\partial K}(K)\to \Omega,\quad \phi\mto\phi-\id_K
\]
as a global chart.
As our main result, we obtain (see Sections~\ref{sec4}, \ref{sec5} and \ref{sec6}):\\[4mm]
{\bf Theorem A.} \emph{$\Diff_{\partial K}(K)$ is a $C^0$-regular
Lie group.}\\[4mm]
Recall that,
if $G$ is a Lie group modelled on a locally convex space~$E$,
with multiplication $\mu\colon G\times G\to G$, then the tangent map
$T\mu\colon T(G\times G)\cong TG\times TG\to TG$ restricts to a smooth right action
\[
TG\times G\to TG, \quad (v,g)\mto v.g
\]
(identifying $G$ with the zero-section in $TG$).
Let $\cg:=L(G):=T_{\bf e}G\cong E$ be the Lie algebra of~$G$ (the tangent space at the neutral element ${\bf e}$).
The Lie group $G$ is called \emph{$C^0$-regular}
if for each $\gamma\in C([0,1],\cg)$, there is a (necessarily unique)
$C^1$-curve $\Evol^r(\gamma):=\eta\colon [0,1]\to G$ such that $\eta(0)={\bf e}$
and
\[
\eta'(t)=\gamma(t).\eta(t)\quad\mbox{for all $t\in[0,1]$,}
\]
and moreover the map
\[
\evol^r\colon C([0,1],\cg)\to G,\quad \gamma\mto\Evol^r(\gamma)(1)
\]
is smooth (using the compact-open topology on the left);
cf.\
\cite{Dah},
\cite{SEM},
\cite{Nee}.\\[2.3mm]
If $G$ is $C^0$-regular, then $G$ is \emph{regular} (i.e., it has the analogous property with $C^\infty([0,1],\cg)$
in place of $C([0,1],\cg)$). Regularity is a central concept in infinite-dimensional
Lie theory, and needed as a hypotheses in many results of this theory. We refer to \cite{Mil} and \cite{Nee}
for more information (cf.\ also \cite{KaM}). Proofs for regularity properties of diffeomorphism
groups can be found, e.g., in
\cite{Ham}, \cite{KaM}, \cite{KMR}, \cite{Mil}, \cite{Shm},
and \cite{Wal}.\\[2.3mm]
Previously, mappings of the form $\phi\mto\phi-\id$ have been used as a global chart
for the Lie group $\Diff_c(\R^n)$ of compactly supported smooth diffeomorphisms
of $\R^n$ \cite{DRN}, for certain weighted diffeomorphism groups of $\R^n$ (like Lie groups of rapidly decreasing diffeomorphisms)
\cite{Wal}, and for further more specialized diffeomorphism groups~\cite{KMR}.\\[2.3mm]
Compared to classical discussions of $\Diff(M)$ for a manifold~$M$,
we encounter the difficulty that neither the inverse function theorem,
nor the implicit function theorem, nor smooth dependence of fixed points on parameters
is readily available in the literature for mappings on sufficiently general non-open sets (like~$K$).
We therefore have to develop such tools as far as required
for our purposes.
We also have to develop a theory of ordinary differential equations
on~$K$, as the
$C^0$-regularity of $\Diff_{\partial K}(K)$ is closely linked
to flows of differential equations on~$K$.
In particular, we find that integral curves for smooth time-dependent vector fields
on~$K$ behave as nicely as in the classical case of a compact smooth manifold
without boundary, as long as the vector fields vanish on $\partial K$
(see Section~\ref{secODE}):\\[4mm]
{\bf Theorem B.} \emph{Let $J\sub\R$ be a non-degenerate interval, $K\sub\R^n$ a compact
convex set with non-empty interior, $P\sub F$ be a convex subset with non-empty interior
in a locally convex space~$F$, and}
\[
f\colon P\times J\times K\to\R^n
\]
\emph{be a smooth function such that $f(p,t,x)=0$ for all $p\in P$, $t\in J$ and $x\in\partial K$.
Then the initial value problem}
\[
y'(t)=f(p,t,y(t)),\quad y(t_0)=x_0
\]
\emph{has a unique solution $y_{p,t_0,x_0}\colon J\to K$ defined on all of~$J$,
for all $p\in P$, $t_0\in J$, and $x_0\in K$.
The associated flow}
\[
P\times J\times J\times K \to K,\quad (p,t_0,t,x_0)\mto y_{p,t_0,x_0}(t)
\]
\emph{is smooth.}\\[4mm]
Beyond $\Diff_{\partial K}(K)$,
for $K$ as before and $r\in\N\cup\{\infty\}$ we consider the
group $\Diff_{\partial K}^{C^r}(K)$ of all $C^r$-diffeomorphisms $\phi\colon K\to K$
with $\phi|_{\partial K}=\id_{\partial K}$. For $r\in\N$,
we make $\Diff_{\partial K}^{C^r}(K)$ a smooth Banach manifold with a global chart
and show that it is a topological group and has smooth right translations
\[
\rho_\phi\colon \Diff_{\partial K}^{C^r}(K)\to\Diff_{\partial K}^{C^r}(K),\quad \psi\mto\psi\circ \phi
\]
for all $\phi\in \Diff^{C^r}_{\partial K}(K)$.
Moreover, for all $r\in\N\cup\{\infty\}$ and $s\in\N_0\cup\{\infty\}$,
the composition map
\[
\Diff_{\partial K}^{C^{r+s}}(K)\times \Diff_{\partial K}^{C^r}(K)\to\Diff_{\partial K}^{C^r}(K),\;\;
(\psi,\phi)\mto\psi\circ\phi
\]
and the inversion map
\[
\Diff_{\partial K}^{C^{r+s}}(K)\to \Diff_{\partial K}^{C^r}(K),\quad\phi\mto\phi^{-1}
\]
are $C^s$ (see Sections \ref{sec3}, \ref{sec4}, and \ref{sec5}).
Such refined information was basic in the ILB-approach to
infinite-dimensional Lie groups (see \cite{Omo} and the references therein).
We shall also see that the map
\begin{equation}\label{diffparbd}
\Diff^{C^r}_{\partial K}(K)\times K\to K,\quad (\phi,y)\mto \phi^{-1}(y)
\end{equation}
is $C^r$, as a special case
of an inverse function theorem with parameters:\\[4mm]
{\bf Theorem C.} \emph{Let $F$ be a locally convex space, $U\sub F$ be a convex subset with
non-empty interior, $r\in\N\cup\{\infty\}$ and $f\colon U\times K\to K$ be a $C^r$-map such that
$f_z:=f(z,\sbull)\in\Diff^{C^r}_{\partial K}(K)$ for all $z\in U$. Then also the following map
is $C^r$}:
\[
g\colon U\times K\to K,\quad (z,y)\mto (f_z)^{-1}(y).
\]
For $K$ as before, let $\Diff_{\flt}(K)$ be the group of all $\phi\in\Diff_{\partial K}(K)$
such that not only $\phi-\id_K$, but also all derivatives of this mapping vanish on $\partial K$.
In Section~\ref{secflat}, we show:\\[4mm]
{\bf Theorem D.}
\emph{$\Diff_{\flt}(K)$ is a $C^0$-regular Lie subgroup of $\Diff_{\partial K}(K)$.}\\[2.3mm]
As a special case of our considerations, we obtain a $C^0$-regular smooth Lie group
structure on the index 2 subgroup $\Diff([0,1])_+:=\Diff_{ \{0,1\} } ([0,1]) \sub \Diff([0,1])$
of all orientation-preserving smooth diffeomorphisms of $[0,1]$.
It is now easy to make also $\Diff([0,1])$
a Lie group with $\Diff([0,1])_+$ as an open submanifold.\footnote{Since $\Diff([0,1])$ is generated by $\Diff([0,1])_+$
and the diffeomorphism $I\colon [0,1]\to[0,1]$, $x\mto 1-x$, it suffices to show that
the map $\Diff([0,1])_+\to \Diff([0,1])_+$, $\phi\mto I\circ \phi\circ I$ (with $I=I^{-1}$)
is smooth, or equivalently, that the map $\Omega\to\Omega$, $\gamma\mto I\circ (\id_{[0,1]}+\gamma)\circ I-\id_{[0,1]}
=:\eta$
is smooth. But $\eta=-\gamma\circ I=-I^*(\gamma)$ is the negative of the
pullback of $\gamma$ along $I$, where $I^*$ a continuous linear (and hence smooth) self-map
of $C^\infty_{\{0,1\}}([0,1])$ (see \cite{GaN}, cf.\ \cite{FUN}).}
The Lie group $\Diff([0,1])_+$ is also of pedagogical interest,
as it is the diffeomorphism group whose Lie group structure is most easily obtained, using only a minimum of analysis and
geometry (already $\Diff(\bS^1)_+$ and $\Diff_c(\R)$ are more complicated to discuss).
The paper places this construction in a larger context
and provides relevant, more difficult additional information (like $C^0$-regularity).\\[2.3mm]
Note that the Lie group structure on $\Diff_{\flt}(K)$ (but not on $\Diff_{\partial K}(K)$) can be
obtained in an alternative fashion, as follows.
Consider the vector subspace $C^\infty_{\flt}(K,\R^n)$ of all $\gamma\colon K\to \R^n$ in $C^\infty_{\partial K}(K,\R^n)$ such that also all
derivatives of~$\gamma$ vanish on~$\partial K$. Then $\gamma$ extends via $0$ to a smooth map $\wt{\gamma}\colon \R^n\to\R^n$
(see, e.g., \cite[Proposition 3.32]{Nik})
which is an element of the weighted function space $C^\infty_\cW(\R^n,\R^n)$ (as defined in \cite{Wal}),
when the set of weights is chosen as $\cW:=\{1,\infty{\bf 1}_{\R^n\setminus K}\}$ where~$1$ is the constant function on $\R^n$ with value~$1$ and ${\bf 1}_{\R^n\setminus K}$ is the characteristic function (indicator function)
of the subset $\R^n\setminus K$ of $\R^n$. It is clear that the map
\[
C^\infty_{\flt}(K,\R^n)\to C^\infty_\cW(\R^n,\R^n),\quad \gamma \mto \wt{\gamma}
\]
is an isomorphism of topological vector spaces, whence
\[
\Diff_{\flt}(K)\to \Diff_\cW(\R^n),\quad \id_K+\gamma\mto \id_{\R^n}+\wt{\gamma}
\]
is an isomorphism of Lie groups, where $\Diff_\cW(\R^n)$ is a special case of the weighted
diffeomorphism groups
\[
\Diff_\cW(E):=\{\phi\in\Diff(E)\colon \phi-\id_E,\phi^{-1}-\id_E\in C^\infty_\cW(E, E)\}
\]
constructed in \cite{Wal}, for $E$ a real Banach space and $\cW$ a set of functions
$f\colon E\to\R\cup\{\pm\infty\}$ which contains the constant function~$1$ on~$E$.
We mention that \cite{Wal} only establishes regularity
for weighted diffeomorphism groups, not $C^0$-regularity.
\section{Preliminaries and notation}\label{sec1}
We write $\N=\{1,2,\ldots\}$ and $\N_0:=\N\cup\{0\}$.
All locally convex (topological real vector) spaces
and all compact topological spaces are assumed Hausdorff.
We write $\graph(f):=\{(x,f(x))\colon x\in X\}\sub X\times Y$
for the graph of a function $f\colon X\to Y$.
If $f\colon X\to Y$ is a function between metric spaces $(X,d_X)$ and $(Y, d_Y)$,
we define
\[
\Lip(f):=\sup\left\{\frac{d_Y(f(x),f(y))}{d_X(x,y)}\colon x\not=y\in X\right\}\in
[0,\infty]
\]
and call $f$ \emph{Lipschitz} if $\Lip(f)<\infty$. If $(E,\|.\|)$ is a Banach space,
we write $\GL(E)$ for the group of continuous automorphisms of
the vector space~$E$.
For $x\in E$ and $r>0$, we write $B^E_r(x):=\{y\in E\colon \|y-x\|<r\}$
and $\wb{B}^E_r(x):=\{y\in E\colon \|y-x\|\leq r\}$.
A subset $U$ of a locally convex space $E$ is called \emph{locally convex}
if, for each $x\in U$, there exists a convex neighbourhood of~$x$ in~$U$
(with respect to the induced topology). If $q$ is a continuous seminorm on $E$,
we write $\wb{B}^q_r(0):=\{x\in E\colon q(x)\leq r\}$
for $r>0$. Given locally convex spaces $E$ and $F$, we write $\cL(E,F)_b$ for the space
of continuous linear mappings from $E$ to $F$, endowed with the topology of
uniform convergence on bounded sets. We abbreviate $\cL(E):=\cL(E,E)$.\\[2.3mm]
We shall use a setting of $C^r$-maps between open subsets of locally convex spaces
which goes back to A. Bastiani~\cite{Bas} and is also known as Keller's $C^r_c$-theory.
See \cite{RES}, \cite{GaN}, \cite{Ham}, \cite{Mic} and \cite{Mil}
for streamlined introductions, cf.\ also \cite{BGN}.
For a discussion of $C^r$-maps on non-open domains (as in \ref{defnonop}),
see~\cite{GaN}.
\begin{numba}\label{defCr}
If $E$ and $F$ are locally convex spaces, $U\sub E$
is open and $r\in\N_0\cup\{\infty\}$, then a map
$f\colon U\to F$ is called $C^r$ if it is continuous,
the iterated directional derivatives
\[
d^{(k)}f(x,y_1,\ldots, y_k):=(D_{y_k}\cdots D_{y_1}f)(x)
\]
exist for all $k\in\N$ such that $k\leq r$,
all points $x\in U$ and all directions $y_1,\ldots , y_k\in E$,
and the maps $d^{(k)}f\colon U\times E^k\to F$ so obtained are
continuous.
If $r\geq 1$, then a map $f$ as before is $C^r$ if and only if $f$ is $C^1$ and $df:=d^{(1)}f\colon U\times E\to F$ is $C^{r-1}$ (see, e.g., \cite{RES} or \cite{GaN}).
\end{numba}
\begin{numba}\label{rulepartdiff}
(Rule on partial differentials).
If $E$, $F$, and $H$ are locally convex spaces, $U\sub E$ and $V\sub F$
open subsets and $f\colon U\times V\to H$ a continuous map,
then $f$ is $C^1$ if and only if the directional derivatives
\[
d_1f(x,y; x_1):=(D_{(x_1,0)}f)(x,y)\quad\mbox{and}\quad
d_2f(x,y;y_1):=(D_{(0,y_1)}f)(x,y)
\]
exist for all $x\in U$, $y\in V$, $x_1\in E$ and $y_1\in F$,
and define continuous functions
$d_1f\colon U\times V\times E\to H$ and $d_2f\colon U\times V\times F\to H$.
In this case,
\[
df((x,y),(x_1,y_1))=d_1f(x,y;x_1)+d_2f(x,y;y_1)
\]
for all $(x,y)\in U\times V$
and $(x_1,y_1)\in E\times F$ (see \cite{RES}).
\end{numba}
\begin{numba}\label{defnonop}
If $U$ is replaced with a locally convex subset $U\sub E$ with dense interior~$U^0$
in \ref{defCr}, then a map $f\colon U\to F$ is called $C^r$
if $f$ is continuous, $f|_{U^0}$ is $C^r$ and $d^{(k)}(f|_{U^0})\colon U^0\times E^k\to F$
has a (necessarily unique) continuous extension $d^{(k)}f\colon U\times E^k\to F$
for all $k\in\N$ such that $k\leq r$.
Then
\[
f^{(k)}(x):=d^{(k)}f(x,\sbull)\colon E^k\to F
\]
is a continuous symmetric $k$-linear map, for each $k$ as before and $x\in U$
(see \cite{GaN}).
We abbreviate $df:=d^{(1)}f$ and $f'(x):=f^{(1)}(x)=df(x,\sbull)$,
which is a continuous linear map from~$E$ to~$F$.
\end{numba}
\begin{numba}\label{defsmooCk}
In the preceding situation, we endow the space $C^r(U,F)$
of all $C^r$-maps $f\colon U\to F$ with the so-called \emph{compact-open $C^r$-topology},
i.e., the initial topology with respect to the linear maps
\[
C^r(U,F)\to C(U\times E^k,F)_{c.o.},\quad f\mto d^{(k)}f
\]
for all $k\in\N_0$ with $k\leq r$ (with $d^{(0)}f:=f$), where the spaces on the right hand side
are endowed with the compact-open topology.
\end{numba}
\begin{numba}
If $J\sub \R$ is a non-degenerate interval and $\gamma \colon J\to E$ a $C^1$-curve,
as usual we write $\gamma'(t):=\frac{d\gamma}{dt}(t)=(D_1\gamma)(t)\in E$
for $t\in J$;
no confusion with $\gamma'\colon J\to \cL(\R,E)$ should arise (as the meaning
will be clear from the context).
\end{numba}
\begin{numba}
Let $E$, $F$, and $H$ be locally convex spaces, $U\sub E$ and $V\sub F$
be locally convex subsets with dense interior, $r,s\in \N_0\cup\{\infty\}$
and $f\colon U\times V\to H$ be a map (more generally, if $r=0$, then $U$
can be any topological space). Following \cite{Alz} and \cite{AaS}, 
$f$ is called a $C^{r,s}$-map if $f$ is continuous, the iterated directional derivatives
\[
d^{(k,\ell)}f(x,y,x_1,\ldots,x_k,y_1,\ldots, y_\ell):=
(D_{(x_k,0)}\cdots D_{(x_1,0)}D_{(0,y_\ell)}\cdots D_{(0,y_1)}f)(x,y)
\]
exist for all $k,\ell\in\N_0$ with $k\leq r$ and $\ell\leq s$
and all $(x,y)\in U^0\times V^0$, $x_1,\ldots, x_k\in E$ and $y_1,\ldots, y_\ell\in F$,
and admit continuous extensions
\[
d^{(k,\ell)}f\colon U\times V\times E^k\times F^\ell\to H.
\]
We shall frequently use the following \emph{exponential law}:
If $f\colon U\times V\to H$ is $C^{r,s}$, then $f^\vee(x):=f(x,\sbull)\in C^s(V,H)$
for all $x\in U$, and the map
\[
f^\vee\colon U\to C^s(V,H),\quad x\mto f(x,\sbull)
\]
is $C^r$ \cite[Theorem 3.25\,(a)]{AaS}.
If $V$ is, moreover, locally compact and a function $g\colon U\to C^s(V,E)$ is $C^r$, then
\[
\wh{g}\colon U\times V\to H,\quad \wh{g}(x,y):=g(x)(y)
\]
is $C^{r,s}$ \cite[Theorem 3.28\,(a)]{AaS}.
An analogous definition of $C^{\alpha_1,\ldots,\alpha_n}$-maps on $n$-fold direct products
is possible for $\alpha_1,\ldots,\alpha_n\in\N_0\cup\{\infty\}$, and again an exponential law is available
(see \cite{Alz}).
\end{numba}
Vector-valued integrals depend continuously on parameters (see, e.g., \cite{GaN}).
\begin{numba}\label{pardepint}
Let $U$ be a topological space, $E$ be a locally convex space
and\linebreak
$f\colon U\times [a,b]\to E$ be a continuous map such that the weak integral
\[
g(x):=\int_a^b f(x,t)\, dt
\]
exists in $E$ for each $x\in U$. Then $g\colon U\to E$ is continuous.
\end{numba}
\begin{numba}
If $K$ is a compact topological space and $(F,\|.\|_F)$ a Banach space, we write
$\|.\|_\infty$ for the supremum norm on $C(K,F)$ given by
\[
\|\gamma\|_\infty:=\sup_{x\in K}\|\gamma(x)\|_F\quad\mbox{for $\gamma\in C(K,F)$,}
\]
which defines the compact-open topology on $C(K,F)$.\\[2.3mm]
If $(F,\|.\|_F)=(\cL(E_1,E_2),\|.\|_{op})$ is a space of continuous linear maps between two Banach spaces $(E_1,\|.\|_1)$
and $(E_2,\|.\|_2)$ with the operator norm, we write $\|\gamma\|_{\infty,op}$ instead of $\|.\|_\infty$, for emphasis.\\[2.3mm]
If $F$ is a locally convex space and $q$ a continuous seminorm on~$F$,
we define a seminorm $\|.\|_{\infty,q}$ on $C(K,F)$ via
\[
\|\gamma\|_{\infty, q}:=\sup_{x\in K}q(\gamma(x)).
\]
The compact-open topology on $C(K,F)$ is defined by the set of all $\|.\|_{\infty,q}$.
\end{numba}
\begin{numba}\label{Dcts}
If $E:=\R^n$ and $K\sub E$ is a compact convex subset with non-empty interior,
then the map
\[
D\colon C^1(K,E)\to C(K,\cL(E)),\quad \gamma\mto\gamma'
\]
is continuous linear.
This follows from the observation that
\[
C^1(K,E)\to [0,\infty[\, ,\quad \gamma\mto \|\gamma'\|_{\infty,op}=\sup_{(x,y)\in L}\|d\gamma(x,y)\|
\]
is a continuous seminorm on $C^1(K,E)$ since
$C^1(K,E)\to C(K\times E,E)$, $\gamma\mto d\gamma$ is a continuous linear map and
$L:=K\times \wb{B}^R_1(0)$ is compact.
\end{numba}
\begin{la}\label{C2fprime}
Let $E$ and $F$ be locally convex spaces, $U\sub E$ be a locally convex subset
with dense interior and $f\colon U\to F$ be a $C^2$-map.
Then the map $f'\colon U\to\cL(E,F)_b$, $x\mto f'(x)=df(x,\sbull)$ is continuous.
\end{la}
\begin{proof}
Let $q$ be a continuous seminorm on $F$ and $B\sub E$ be a bounded set.
By continuity of $d^{(2)}f\colon U\times E\times E\to F$,
there is a convex neighbourhood $V\sub U$ of $x$ and a continuous seminorm $p$ on $E$
such that
\[
d^{(2)}f (V\times \wb{B}^p_1(0)\times \wb{B}^p_1(0))\sub\wb{B}^q_1(0).
\]
Then $rB\sub \wb{B}^p_1(0)$ for some $r>0$.
After shrinking $V$, we may assume that $V-x\sub \wb{B}^p_r(0)$.
For all $y\in V$ and $b\in B$, we deduce that
\begin{eqnarray*}
q((f'(y)-f'(x))(b)) &=& q(df(y,b)-df(x,b))\\
&=& q\left(\int_0^1 d^{(2)}f(x+t(y-x),b,y-x)\,dt\right)\\
&\leq & \int_0^1 \underbrace{q\Big(d^{(2)}f\Big(x+t(y-x),rb,\frac{1}{r}(y-x)\Big)\Big)}_{\leq 1}\, dt\leq 1
\end{eqnarray*}
and thus $(f'(y)-f'(x))(B)\sub \wb{B}^q_1(0)$.
\end{proof}
\begin{numba}
If $K$ is a compact convex subset of $E:=\R^n$ with non-empty interior and $r\in\N\cup\{\infty\}$,
we let $\Diff^{C^r}_{\partial K}(K)$ be the group of all $C^r$-diffeomorphisms~$\phi$ of~$K$
such that $\phi|_{\partial K}=\id_{\partial K}$.
We let $C^r_{\partial K}(K,E)\sub C^r(K,E)$
be the closed vector subspace of all $\gamma\in C^r(K,E)$ such that $\gamma|_{\partial K}=0$.
We abbreviate
$\Diff_{\partial K}(K):=\Diff^{C^\infty}_{\partial K}(K)$.
\end{numba}
\section{Auxiliary results}\label{sec2}
The following fact is well-known; the simple proof is recalled in Appendix~\ref{appA}.
\begin{la}\label{clogra}
Let $f\colon X\to K$ be a map from a topological space $X$
to a Hausdorff topological space $K$.
\begin{itemize}
\item[\rm(a)]
If $f$ is continuous, then its graph $\graph(f)$ is closed in $X\times K$.
\item[\rm(b)]
If $K$ is compact, then $f$ is continuous if and only if $\graph(f)$ is closed.
\end{itemize}
\end{la}
\begin{la}\label{autctsinv}
Let $X$ be a topological space, $K$ a compact topological space and
$f\colon X\times K\to K$ be a continuous mapping such that
\[
f_x:=f(x,\sbull)\colon K\to K,\quad y\mto f(x,y)
\]
is bijective for each $x\in K$.
Then the map
\[
g\colon X\times K\to K,\quad (x,z)\mto (f_x)^{-1}(z)
\]
is continuous and
$h\colon \!X\!\times \! K\!\to\! X\! \times \! K$, $(x,y)\mto (x,f(x,y))$
a homeomorphism.
\end{la}
\begin{proof}
Because $f$ is continuous, $\graph(f)$ is closed in $X\times K\times K$
(see Lemma~\ref{clogra}).
Since
\[
\sigma\colon X\times K\times K\to X\times K\times K,\quad (x,y,z)\mto (x,z,y)
\]
is a homeomorphism, also $\graph(g)=\sigma(\graph(f))$ is closed
in $X\times K\times K$. Hence $g$ is continuous, by Lemma~\ref{clogra}\,(b).
It is immediate from the definition that the continuous map~$h$ is a bijection with inverse
\[
h^{-1}(x,z)=(x,(f_x)^{-1}(z))=(x,g(x,z)).
\]
As $g$ is continuous, $h^{-1}$ is continuous and thus $h$ is a homeomorphism.
\end{proof}
The following lemma generalizes a special case stated in \cite[Theorem 5.3.1]{Ham}.
\begin{la}\label{invlinpar}
Let $E$ and $F$ be locally convex spaces, $U\sub E$ a locally convex
subset with dense interior, $k\in \N\cup\{\infty\}$
and
\[
f\colon U\times F\to F
\]
be a $C^k$-map such that $f_x:=f(x,\sbull)\colon F\to F$
is linear and bijective for each $x\in U$ and the map
\[
g\colon U\times F\to F,\quad g(x,z):=(f_x)^{-1}(z)
\]
is continuous. Then $g$ is $C^k$.
\end{la}
\begin{proof}
Since $g(x,\sbull)$ is continuous linear for $x\in U^0$, we have
\[
d_2g(x,z; w)=g(x,w)
\]
for all $x\in U^0$, $z,w\in F$, and the same formula defines a
continuous extension for $(x,z,w)\in U\times F\times F$.
Let $x\in U^0$, $y\in E$, $z\in F$ and $0\not=t\in \R$ with $x+ty\in U^0$.
Then
\begin{eqnarray}
f_x\left(\frac{g(x+ty,z)-g(x,z)}{t}\right) &=& \frac{f_x(g(x+ty,z))-z}{t}\label{brustohrauge}\\
&=&
\frac{(f_x-f_{x+ty})(g(x+ty,z))}{t}\notag \\
&=&-\int_0^1 d_1f(x+sty,g(x+ty,z);y)\,ds,\label{theintgal}
\end{eqnarray}
using that $f_x(g(x,z))=z=f_{x+ty}(g(x+ty,z))$.
By \ref{pardepint},
the integral in~(\ref{theintgal}) converges to $-\int_0^1 d_1f(x,g(x,z);y)\, ds=-d_1f(x,g(x,z); y)$
as $t\to 0$. Applying now the continuous linear map $(f_x)^{-1}$ to (\ref{brustohrauge}) and (\ref{theintgal}),
we see that
\begin{eqnarray*}
\frac{g(x+ty,z)-g(x,z)}{t} & \to&  -(f_x)^{-1}(d_1f(x,g(x,z);y))\\
&=&-g(x,d_1f(x,g(x,z);y))
\end{eqnarray*}
as $t\to 0$. Thus $d_1g(x,z; y)=g(x,d_1f(x,g(x,z);y))$
and the same formula defines a continuous $F$-valued function of $(x,z,y)\in U\times F\times E$.
Using~\ref{rulepartdiff}, we find that $g$ is $C^1$, with
\[
dg(x,z,y,w)=d_1g(x,z;y)+d_2g(x,z;w)=g(x,w)-g(x,d_1f(x,g(x,z); y))
\]
for $(x,z,y,w)\in U\times F\times E\times F$. Now if $g$ is $C^{k-1}$ by induction,
then the preceding formula shows that also $dg$ is $C^{k-1}$, and thus $g$ is $C^k$.
\end{proof}
We shall use a result on the parameter-dependence of fixed points.
\begin{la}\label{pardepfix}
Let $E$ be a locally convex space, $(F,\|.\|)$ a Banach space,
$P\sub E$ and $V\sub F$ be locally convex subsets with dense interior,
$k\in\N_0\cup\{\infty\}$
and
\[
f\colon P\times V\to F
\]
be a $C^k$-map which defines a ``uniform family of contractions"
in the sense that $f_p:=f(p,\sbull)\colon V\to F$ is Lipschitz for each $p\in P$ with
\[
\theta:=\sup_{p\in P}\Lip(f_p)\;<1.
\]
Assume that $f_p$ has a $($necessarily unique$)$ fixed point $x_p$ for each $p\in P$
and define $\phi\colon P\to V$, $\phi(p):=x_p$.
Then the following holds:
\begin{itemize}
\item[\rm(a)]
If $V$ is open,
then $\phi\colon P\to V$ is $C^k$.
\item[\rm(b)]
If $\phi\colon P\to V$ is continuous and $k\geq 2$, then $\phi$ is $C^k$.
\item[\rm(c)]
If $\phi$ is continuous, $f$ is $C^1$
and the map
\begin{equation}\label{byhand}
g\colon P\times V\to (\cL(F),\|.\|_{op}),\quad (p,x)\mto d_2f(p,x;\sbull)
\end{equation}
is continuous, then $\phi$ is $C^1$.
\end{itemize}
\end{la}
\begin{proof}
(a) is a special case of \cite[Theorem D]{FRE}.

(b) and (c). It suffices to consider $k\in \N$. Let $\lambda\colon F\to E\times F$, $x\mto (0,x)$ be the continuous linear inclusion
map. Then the linear mapping
\[
\cL(\lambda, F)\colon \cL(E\times F,F)\to\cL(F,F),\quad A\mto A\circ \lambda
\]
is continuous.
If $f$ is $C^2$, then the map $g=\cL(\lambda,F)\circ f'$ from (\ref{byhand}) is continuous
as a consequence of Lemma~\ref{C2fprime}.
We may therefore assume continuity of~$g$ now and prove the assertion by induction on
$k\in\N$.
Note that
\[
\|d_2f(p,x;\sbull)\|_{op}\leq\theta < 1
\]
since $\Lip(f_p)\leq\theta$. Hence $\id_F-d_2f(p,x;\sbull)\in\GL(F)$
for all $(p,x)\in P\times V$. The inversion map $\GL(F)\to\GL(F)$, $A\mto A^{-1}$ being continuous, also
\[
P\times V\to \cL(F),\quad (p,x)\mto (\id_F-g(p,x))^{-1}
\]
is continuous and hence also the map
\begin{equation}\label{firstdramp}
P\times V\times F\to F,\quad (p,x,z)\mto (\id_F-g(p,x))^{-1}(z),
\end{equation}
as the evaluation map $\cL(F)\times F\to F$, $(A,x)\mto A(x)$ is continuous.
Since
\[
h\colon (P\times V)\times F\to F,\quad (p,x,y)\mto y-d_2f(p,x;y)
\]
is $C^{k-1}$, writing $h_{p,x}:=h(p,x,\sbull)$ we deduce with Lemma~\ref{invlinpar} that the map
\[
\Theta \colon P\times V\times F\to F,\quad (p,x,z)\mto (h_{p,x})^{-1}(z)
\]
(which coincides with the map in (\ref{firstdramp}))
is~$C^{k-1}$.
We know from (a) that $\phi|_{P^0}$ is $C^k$.
Now
\begin{equation}\label{bothsid}
f(p,\phi(p))=\phi(p).
\end{equation}
Using the Chain Rule, we can form the directional derivative at $p\in P^0$ in a direction
$q\in E$ on both sides of~(\ref{bothsid}) and obtain the identity
\[
d_1f(p,\phi(p);q)+d_2f(p,\phi(p);d\phi(p,q))=d\phi(p,q),
\]
which can be solved for $d\phi(p,q)$:
\[
d\phi(p,q)=(\id_F-d_2f(p,\phi(p);\sbull)^{-1}(d_1f(p,\phi(p);q)=\Theta  (p,\phi(p),d_1f(p,\phi(p);q)).
\]
Note that term on the right hand side also defines a continuous $F$-valued function
of $(p,q)\in P\times E$. Hence $\phi$ is $C^1$ with
\[
d\phi(p,q)=\Theta (p,\phi(p),d_1f(p,\phi(p);q)).
\]
Now if $\phi$ is $C^{k-1}$ by induction, then the previous identity shows that
also $d\phi$ is $C^{k-1}$, and thus $\phi$ is $C^k$.
\end{proof}
\section{Submanifold structure and global chart}\label{sec3}
Let $K$ be a compact convex subset with non-empty interior in $E:=\R^n$
and $r\in \N\cup\{\infty\}$.
We show:
\begin{la}\label{issubmfd}
$\Diff^{C^r}_{\partial K}(K)$ is a smooth submanifold of
$C^r(K,E)$ modelled on the closed vector subspace $C^r_{\partial K}(K,E)$
of $C^r(K,E)$. Moreover, $\Diff^{C^r}_{\partial K}(K)$ admits a global chart.
\end{la}
\begin{proof}
The affine vector subspace $\id_K+C^r_{\partial K}(K,E)$
is a $C^\infty$-submanifold of $C^r(K,E)$ modelled on the closed vector subspace
$C^r_{\partial K}(K,E)$ of $C^r(K,E)$, as
\[
\Psi\colon C^r(K,E)\to C^r(K,E),\quad \phi\mto \phi-\id_K
\]
is a $C^\infty$-diffeomorphism (and hence a global chart for $C^r(K,E)$)
such that
\[
\Psi(\id_K+C^r_{\partial K}(K,E))=C^r_{\partial K}(K,E).
\]
We claim that
\[
\Omega_r:=\Diff^{C^r}_{\partial K}(K)-\id_K
\]
is open in $C^r_{\partial K}(K,E)$. If this is true, then $\Diff^{C^r}_{\partial K}(K)$ will be (relatively) open
in the submanifold $\id_K+C^r_{\partial K}(K,E)$ of $C^r(K,E)$, entailing that also $\Diff^{C^r}_{\partial K}(K)$
is a submanifold of $C^r(K,E)$ modelled on $C^r_{\partial K}(K,E)$.
Moreover,
\[
\Phi_r\colon \Diff^{C^r}_{\partial K}(K)\to\Omega_r,\quad \phi\mto\phi-\id_K
\]
will be a global chart for $\Diff^{C^r}_{\partial K}(K)$ (establishing the lemma).
To prove the claim,
let $x_0\in K^0$.
Fix a norm $\|.\|$ on $E=\R^n$ and write
\[
\|\eta\|_{\infty,op}:=\sup_{x\in K}\|\eta(x)\|_{op}
\]
for continuous functions $\eta\colon K\to \cL(E)$.
Let
\[
Q_r\sub C^r_{\partial K}(K,E)
\]
be the set of all $\gamma\in C^r_{\partial K}(K,E)$ such that
$\gamma'(K)\sub  \GL(E)-\id_E$ holds,
$x_0+\gamma(x_0)  \in  K^0$, and
$\id_K+\gamma$ is injective.
It is clear that
$\phi-\id_K\in Q_r$
for each $\phi\in \Diff^{C^r}_{\partial K}(K)$
and thus
\[
\Omega_r\sub Q_r.
\]
To see that $\Omega_r=Q_r$,
let $\gamma\in Q_r$. Then $\phi:=\id_K+\gamma\colon K\to E$
is injective. We show that $\phi\in\Diff^{C^r}_{\partial K}(K)$
(whence $\gamma=\phi-\id_K\in\Omega_r$).
Since
\[
\phi'(x)=\id_E+\gamma'(x)\in\GL(E)
\]
for each $x\in K^0$ and $\phi|_{K^0}$ is injective,
we deduce from the inverse function theorem that $\phi(K^0)$ is open in~$E$
and $\phi|_{K^0}\colon K^0\to\phi(K_0)$ is a $C^r$-diffeomorphism.
Since $\gamma|_{\partial K}=0$, we have $\phi|_{\partial K}=\id_{\partial K}$
and hence
\[
\phi(\partial K)=\partial K.
\]
As $\phi$ is injective, this implies that
\[
\phi(K^0)\cap\partial K=\emptyset.
\]
Therefore $K^0$ is the disjoint union of the open set $\phi^{-1}(E\setminus K)$
and the open set $\phi^{-1}(K^0)$ which is non-empty as $\phi(x_0)\in K^0$ by definition of~$Q_r$.
Since $K^0$ is convex and hence connected, we deduce that $\phi^{-1}(E\setminus K)=\emptyset$
and hence
\[
\phi(K^0)\sub K^0.
\]
Now $\phi(K^0)$ is open in $K^0$ but it is also (relatively) closed in~$K^0$
since
\[
\phi(K^0)=\phi(K^0)\cup \underbrace{K^0\cap \phi(\partial K)}_{=\emptyset}=
K^0\cap \phi(K)
\]
where $\phi(K)$ is compact and hence closed in~$E$.
Since $\phi(K^0)$ is non-empty, the connectedness of $K^0$ implies that $\phi(K^0)=K^0$.
Thus $\phi$ is surjective and hence $\phi$ (being also injective) is a bijection.
Since $K$ is compact, the continuous bijection $\phi\colon K\to K$ is a homeomorphism.
As $\phi$ is $C^r$, the map $d\phi\colon K\times E\to E$ is $C^{r-1}$
and hence $C^{r-1,0}$ (cf.\ \cite[Lemma 3.17]{AaS}),
entailing that
\[
\phi'=(d\phi)^\vee\colon K\to \GL(E)\sub C(E,E)
\]
is $C^{r-1}$ (see \cite[Theorem 3.25\,(a)]{AaS}).
We know that
\[
\phi^{-1}|_{K^0}=(\phi|_{K^0})^{-1}
\]
is $C^r$ and
\[
(\phi^{-1})'(y)=(\phi'(\phi^{-1}(y)))^{-1}
\]
for each $y\in K^0$.
As the inversion map $\GL(E)\to \GL(E)$ is smooth and hence continuous, we deduce that
the map
\begin{equation}\label{concretef}
\psi\colon K\to \GL(E),\quad y\mto (\phi'(\phi^{-1}(y)))^{-1}
\end{equation}
is continuous. Because the evaluation map $\ve\colon \cL(E)\times E\to E$ is continuous bilinear
(and thus smooth, for later use), we see that also the map
\[
\wh{\psi}\colon K\times E\to E,\quad (y,z)\mto \psi(y)(z)=\ve(\psi(y),z)
\]
is continuous. By the preceding, $\wh{\psi}\colon K\times E\to E$ is a continuous extension of
$d(\phi^{-1}|_{K^0})$. Hence $\phi^{-1}$ is $C^1$ with $d(\phi^{-1})=\wh{\psi}$.
If now $\phi^{-1}$ is $C^{k-1}$ by induction for an integer $2\leq k\leq r$,
then (\ref{concretef}) and the Chain Rule entail that~$\psi$ is~$C^{k-1}$.
Hence $\wh{\psi}=d(\phi^{-1})$ is $C^{k-1,0}$ (see \cite[Theorem 3.28]{AaS}), whence $d(\phi^{-1})$ is $C^{k-1,\infty}$
by linearity in the second argument (see \cite[Lemma 3.14]{AaS}). Thus $d(\phi^{-1})$ is $C^{k-1}$, by
\cite[Lemma 3.15]{AaS},
entailing that~$\phi^{-1}$ is $C^k$ and thus~$C^r$.
Hence $\phi\in\Diff_{\partial K}^{C^r}(K)$, as desired.\\[2.3mm]
To see that $\Omega_r$ is open, note that
the point evaluation $\ve_{x_0}\colon C^r_{\partial K}(K,E)\to E$, $\gamma\mto\gamma(x_0)$ is continuous
linear.
Moreover, the map
\[
C^r(K,E)\to C(K,\cL(E)),\quad \gamma\mto \gamma'
\]
is continuous linear and $\GL(E)-\id_E$ is open
in $\cL(E)$. Hence
\[
U :=\{\gamma\in C^r_{\partial K}(K,E)\colon \gamma'(K)\sub \GL(E)-\id_E\;\mbox{and}\;
\gamma(x_0)\in K^0-x_0\}
\]
is open in $C^r_{\partial K}(K,E)$.
Let $\gamma\in \Omega_r$; then $\phi:=\id_K+\gamma\in \Diff^{C^r}_{\partial K}(K)$
and thus
\[
a:= \Lip(\phi^{-1}) = \|(\phi^{-1})'\|_{\infty,op}<\infty.
\]
Now
\[
V:=\Big\{\eta\in U\colon \|\eta'-\gamma'\|_{\infty,op}<\frac{1}{2a}\Big\}
\]
is an open neighbourhood of $\gamma$ in~$U$ (and hence also in $C^r_{\partial K}(K,E)$).
We show that $V\sub Q_r=\Omega_r$ (completing the proof that $\Omega_r$ is open).
The estimate
\[
\|y-x\|=\|\phi^{-1}(\phi(y))-\phi^{-1}(\phi(x))\|\leq\Lip(\phi^{-1})\|\phi(y)-\phi(x)\|=a\|\phi(y)-\phi(x)\|
\]
implies that
$\|\phi(y)-\phi(x)\|\geq \frac{1}{a}\|y-x\|$ for all $x,y\in K$.
For $\eta\in V$ and $x,y\in K$ with $x\not=y$, this entails
\begin{eqnarray*}
\frac{1}{a}\|y-x\|&\leq & \|\phi(y)-\phi(x)\|\\
&\leq& \|(\id_K+\eta)(y)-(\id_K+\eta)(x)\|+\|(\eta-\gamma)(y)-(\eta-\gamma)(x)\|\\
&\leq & \|(\id_K+\eta)(y)-(\id_K+\eta)(x)\|+\Lip(\eta-\gamma)\|y-x\|\\
&\leq & \|(\id_K+\eta)(y)-(\id_K+\eta)(x)\|+\frac{1}{2a}\|y-x\|,
\end{eqnarray*}
using that $\Lip(\eta-\gamma)=\|\gamma'-\eta'\|_{\infty,op}<\frac{1}{2a}$. Hence
\[
\|(\id_K+\eta)(y)-(\id_K+\eta)(x)\|\geq \frac{1}{2a}\|y-x\|>0
\]
and thus $x+\eta(x)\not= y+\eta(y)$. Thus $\id_K+\eta$ is injective, whence
$\eta\in Q_r$ and hence $V\sub Q_r$.
\end{proof}
\section{Smoothness of composition}\label{sec4}
In this section, we establish continuity and differentiability properties
of the composition map $\Diff^{C^r}_{\partial K}(K) \times \Diff^{C^r}_{\partial K}(K)\to \Diff^{C^r}_{\partial K}(K)$,
and related results.
\begin{la}\label{premultHamz}
Let $K$ be a compact convex subset of $E:=\R^n$
with non-empty interior, $V$ a closed convex subset of a finite-dimensional
vector space~$E_2$ with non-empty interior, and $F$ be a locally convex space.
For all $r,s\in \N_0\cup\{\infty\}$,
\[
C^r(K,V):=\{\eta\in C^r(K,E_2)\colon\eta(K)\sub V\}
\]
is a convex subset of $C^r(K,E_2)$ with non-empty interior
and
\[
\Gamma\colon C^{r+s}(V,F)\times C^r(K,V)\to C^r(K,F),\quad \Gamma(\gamma,\eta):=\gamma\circ\eta
\]
is a $C^s$-mapping. If $s\geq 1$, then
\begin{equation}\label{dGamma}
d\Gamma(\gamma,\eta;\gamma_1,\eta_1)=\gamma_1\circ\eta +d\gamma\circ (\eta,\eta_1)
\end{equation}
for all $\gamma,\gamma_1\in C^{r+s}(V,F)$, $\eta\in C^r(K,V)$, and $\eta_1\in C^r(K,E_2)$.
In particular,
\[
\mu\colon \Diff^{C^{r+s}}_{\partial K}(K) \times \Diff^{C^r}_{\partial K}(K)\to \Diff^{C^r}_{\partial K}(K),\quad (\gamma,\eta)
\mto\gamma\circ\eta
\]
is $C^s$, for all $r\in\N\cup\{\infty\}$ and $s\in\N_0\cup\{\infty\}$. For each $\eta\in\Diff^{C^r}_{\partial K}(K)$,
the right translation map $\rho_\eta\colon \Diff^{C^r}_{\partial K}(K)\to\Diff^{C^r}_{\partial K}(K)$,
$\gamma\mto\gamma\circ\eta$ is smooth.
\end{la}
\begin{proof}
It is clear that $C^r(K,V)$ is convex and $C^r(K,V^0)\sub C^r(K,V)^0$.
Actually, $C^r(K,V^0)=C^r(K,V)^0$ holds.\footnote{Suppose there is $\eta\in C^r(K,V)^0\setminus C^r(K,V^0)$; then $\eta(x_0)\in V\setminus V^0$
for some $x_0\in K$. Now $\eta+c\in C^r(K,V)^0$ for all $c$ in a $0$-neighbourhood $U\sub E_2$
(identified with the corresponding constant function $x\mto c$).
Evaluating at $x_0$, we see that $\eta(x_0)+U\sub V$ and hence $\eta(x_0)\in V^0$, contradiction.}\\[2mm]
By \cite[Theorem 3.25\,(a)]{AaS}, the map $\Gamma$ will be $C^s$ if we can show that the map
\[
\wh{\Gamma}\colon (C^{r+s}(V,F)\times C^r(K,V))\times K\to F,\;
\wh{\Gamma}(\gamma,\eta,x):=\Gamma(\gamma,\eta)(x)=\gamma(\eta(x))
\]
is $C^{s,r}$. Now
\begin{equation}\label{cheapheute}
\wh{\Gamma}(\gamma,\eta,x)=\ve_1(\gamma,\ve(\eta,x))
\end{equation}
where $\ve\colon C^r(K,E_2)\times K\to E_2$, $\ve(\eta,x):=\eta(x)$ is $C^{\infty,r}$
and the corresponding evaluation map $\ve_1\colon C^{r+s}(V,F)\times V\to F$ is $C^{\infty, r+s}$.
Applying the Chain Rule in the form \cite[Lemma~81]{Alz} to the right hand side
of (\ref{cheapheute}), we see that $\wh{\Gamma}$ is $C^{\infty,s,r}$
in $(\gamma,\eta,x)$ and hence $C^{s,r}$ in $((\gamma,\eta),x)$, by \cite[Lemma~77]{Alz}.\\[2.3mm]
Since $\Gamma(\gamma,\eta)$ is a continuous linear function of $\gamma\in C^{r+s}(V,F)$ for fixed
$\eta\in C^r(K,V)$,
the first partial differential always exists and is given by
\begin{equation}\label{firstd}
d_1\Gamma(\gamma,\eta;\gamma_1)=\Gamma(\gamma_1,\eta)=\gamma_1\circ\eta.
\end{equation}
If $s\geq 1$ and $\gamma\in C^{r+s}(V,F)$, then
\[
\Gamma(\gamma,\sbull)=C^r(K,\gamma)\colon C^r(K,V)\to C^r(K,F),\quad \eta\mto\gamma\circ\eta
\]
with $dC^r(K,\gamma)=C^r(K,d\gamma)$ (identifying the space $C^r(K,E_2\times E_2)$ with $C^r(K,E_2)\times
C^r(K,E_2)$), see, e.g., \cite{GaN}, cf.\ also \cite{FUN}.
Hence
\[
d\Gamma(\gamma,\eta,\gamma_1,\eta_1)=d_1\Gamma(\gamma,\eta;\gamma_1)+d_2\Gamma(\gamma,\eta;\eta_1)
=\gamma_1\circ\eta+d\gamma\circ(\eta,\eta_1),
\]
establishing (\ref{dGamma}). Since $\Diff^{C^r}_{\partial K}(K)$ is a submanifold of~$C^r(K,E)$,
it suffices to show that~$\mu$ is $C^s$ as a map to $C^r(K,E)$.
But this is the restriction of $\Gamma$ (for $E_2:=F:=E$ and $V:=K$)
to the submanifold $\Diff_{\partial K}^{C^{r+s}}(K)\times \Diff_{\partial K}^{C^r}(K)$
of $C^{r+s}(K,E)\times C^r(K,E)$ and hence $C^s$ (like $\Gamma$).
For each $\eta\in \Diff^{C^r}_{\partial K}(K)$,
the map
$\Gamma(\sbull,\eta)\colon \gamma\mto \gamma\circ \eta$
is continuous linear and hence smooth, entailing that also $\rho_\eta=\Gamma(\sbull,\eta)|_{\Diff^{C^r}_{\partial K}(K)}$
is smooth.
\end{proof}
\section{Smoothness of inversion and Theorem C}\label{sec5}
We prove Theorem~C. As a byproduct,
we obtain continuity and differentiability properties for the inversion
map $\Diff^{C^r}_{\partial K}(K)\to \Diff^{C^r}_{\partial K}(K)$, $\phi\mto\phi^{-1}$.\\[2.3mm]
{\bf Proof of Theorem C.} As before, $E:=\R^n$ and $K\sub E$.
We show by induction that $g$ is~$C^k$ for all $k\in\N_0$ with $k\leq r$.
Since $g$ is continuous by Lemma~\ref{autctsinv}, the case $k=0$ is settled.
Now assume that $k\geq 1$.
Since $f(z,\sbull)\colon K^0\to K^0$ is a $C^r$-diffeomorphism
for each $z\in U^0$, we deduce that
\[
\psi\colon U^0 \times K^0\to U^0  \times K^0,\quad \psi(z,x):=(z,f(z,x))
\]
is a bijection.
Since $d_2f(z,x;\sbull)=(f_z)'(x)\in\GL(E)$ for each $z\in U^0$
and $x\in K^0$, the Inverse Function Theorem with Parameters
(in the form
\cite[Theorem 2.3\,(c)]{IMP}) shows that
$\psi$ is a $C^r$-diffeomorphism (being bijective and a local $C^r$-diffeomorphism around each point).
Note that
\[
d\psi(z,x,z_1,x_1)=(z_1,d_1f(z,x;z_1)+d_2f(z,x;x_1))
\]
for all $z\in U^0$, $z_1\in F$, $x\in K^0$, and $x_1\in E$
(using the Rule on Partial Differentials).
Recalling the familiar identity
\[
(\psi^{-1})'(z,y)=(\psi'(\psi^{-1}(z,y)))^{-1},
\]
we deduce that
\begin{equation}\label{twocompo}
(\psi^{-1})'(z,y,z_1,y_1)=(z_1, ((f_z)'(x))^{-1} (y_1-d_1f(z,x;z_1)))
\end{equation}
for all $z\in U^0$, $z_1\in F$, $y\in K^0$, and $y_1\in E$,
with $x:=(f_z)^{-1}(y)=g(z,y)$.
Since $\psi^{-1}(z,y)=(z,g(z,y))$,
looking at the second component of (\ref{twocompo}) we see that $g|_{U^0\times K^0}$
is $C^r$ (like $\psi^{-1}$), with
\begin{equation}\label{soonhave}
d(g|_{U^0\times K^0})(z,y,z_1,y_1)=((f_z)'(g(z,y)))^{-1}(y_1-d_1f(z,g(z,y); z_1))
\end{equation}
for all $z\in U^0$, $z_1\in F$, $y\in K^0$, and $y_1\in E$.
Let $I\colon\GL(E)\to \GL(E)$, $A\mto A^{-1}$ be the smooth (and hence continuous)
inversion map.
The mapping
\[
h\colon U\times K\times E\to E,\quad h(z,y,v):=d_2f(z,g(z,y);v)
\]
is continuous and $h_{z,y}:=h(z,y,\sbull)=(f_z)'(g(z,y))\in\GL(E)$ for all $(z,y)\in U\times K$.
Hence $h^\vee\colon U\times K\to \GL(E)\sub C(E,E)$, $h^\vee(z,y):=h_{z,y}$ is continuous.
As a consequence, $I\circ h^\vee\colon U\times K\to\GL(E)$ is continuous, whence also the map
\[
\beta :=\wh{I\circ h^\vee}\colon U\times K\times E\to E,\quad (z,y,w)\mto h_{z,y}^{-1}(w)
\]
is continuous.
Note that the right hand side of (\ref{soonhave}) also makes sense
for $z\in U$, $z_1\in F$, $y\in K$ and $y_1\in E$,
and defines a function $\theta\colon U \times K\times F \times E\to E$,
\begin{eqnarray}
\theta(z,y,z_1,y_1)& := &
((f_z)'(g(z,y)))^{-1}(y_1-d_1f(z,g(z,y;z_1)))\notag\\
&=& \beta(z,y,y_1-d_1f(z,g(z,y);z_1))\label{hencegood}
\end{eqnarray}
which is continuous.
As this mapping extends $d(g|_{U^0\times K^0})$ by (\ref{soonhave}),
we see that $g$ is $C^1$ with $dg=\theta$.
If now $k\geq 2$ and $g$ is $C^{k-1}$ by induction,
then $h$ is $C^{k-1}$ and thus $\beta$ is $C^{k-1}$,
by Lemma~\ref{invlinpar}.
Hence $dg=\theta$ is $C^{k-1}$ (see (\ref{hencegood})) and thus $g$ is $C^k$,
which completes the inductive proof.\;\Punkt\vspace{2.3mm}

\noindent
For each $r\in\N\cup\{\infty\}$, the evaluation map $C^r(K,E)\times K\to E$ is $C^{\infty,r}$
and hence $C^r$ (see \cite[Proposition~3.20 and Lemma~3.15]{AaS}), whence its~restriction
\[
f\colon\Diff^{C^r}_{\partial K}(K)\times K\to K,\quad f(\phi,x):=\phi(x)
\]
is $C^r$ as well. Since $f(\phi,\sbull)=\phi\in\Diff^{C^r}_{\partial K}(K)$, Theorem~C shows that
the mapping in (\ref{diffparbd}) is $C^r$.
If $r\in\N\cup\{\infty\}$ and $s\in\N_0\cup\{\infty\}$, then the map
\[
h\colon \Diff^{C^{r+s}}_{\partial K}(K)\times K\to K,\quad (\phi,y)\mto\phi^{-1}(y)
\]
is $C^{r+s}$ by the preceding and hence $C^{s,r}$, entailing that
\[
h^\vee\colon \Diff^{C^{r+s}}_{\partial K}(K)\to C^r(K,E),\quad \phi\mto \phi^{-1}
\]
is $C^s$ as a map to $C^r(K,E)$ (by \cite[Theorem 3.25\,(a)]{AaS})
and hence also as a map to its submanifold $\Diff^{C^r}_{\partial K}(K)$.
\section{Regularity}\label{sec6}
We shall use a standard fact from the theory of ordinary differential equations
(a quantitative version of the Picard-Lindel\"{o}f Theorem):\footnote{The condition $L<1$ is actually superfluous;
one can use the Fixed Point Theorem of Banach-Weissinger~\cite{Wei}
instead of the standard Banach Fixed Point Theorem.}
\begin{la}\label{quantlind}
Let $E$ be a Banach space, $R>0$, $x_0\in E$ and
\[
f\colon [0,1]\times \wb{B}^E_R(x_0)\to E
\]
be a bounded continuous function which satisfies a global Lipschitz condition,
\[
L:=\sup_{t\in[0,1]}\Lip(f(t,\sbull))\; <\;\infty.
\]
If $\|f\|_\infty \leq R$ and $L<1$, then there is a $($unique$)$ $C^1$-function $\gamma\colon [0,1]\to\wb{B}^E_R(x_0)$
such that $\gamma'(t)=f(t,\gamma(t))$ for all $t\in[0,1]$ and $\gamma(0)=x_0$.
\end{la}
\begin{proof}
$A:=\{\eta\in C([0,1],\wb{B}^E_R(x_0))\colon\eta(0)=x_0\}$
is a closed subset of the Banach space $(C([0,1], E),\|.\|_\infty)$ and
\[
\psi\colon A\to A,\quad \psi(\eta)(t):=x_0+\int_0^tf(s,\eta(s))\, ds
\]
is a self-map of $A$. If $L<1$, then $\psi$ is a contraction, and hence $\psi$
has a unique fixed point $\gamma$ (by Banach's Fixed Point Theorem), which (by the Fundamental Theorem of Calculus)
is the unique $C^1$-solution $\gamma$ of the initial value problem described in the
lemma.
\end{proof}
Let $E:=\R^n$ and $K\sub E$ be a compact convex set with non-empty interior.
\begin{la}\label{cmbine}
The following map is smooth:
\[
h\colon C([0,1],C^\infty(K,E))\times C([0,1],K)\to C([0,1],E),\;\;
h(\gamma,\eta)(t):=\gamma(t)(\eta(t)).
\]
\end{la}
\begin{proof}
If we can show that $\wh{h}\colon C([0,1],C^\infty(K,E))\times C([0,1],K)\times[0,1]\to E$, 
$\wh{h}(\gamma,\eta,t):=h(\gamma,\eta)(t)$ is $C^{\infty,0}$ as a function of $((\gamma,\eta),t)$,
then $h=(\wh{h})^\vee$ will be $C^\infty$ by \cite[Theorem~3.25\,(a)]{AaS}.
We have
\[
\wh{h}(\gamma,\eta,t)=\gamma(t)(\eta(t))=\ve_1(\ve_2(\gamma,t),\ve_3(\eta,t))
\]
where the evaluation maps $\ve_3\colon C([0,1],K)\times [0,1]\to K$
and
\[
\ve_2\colon C([0,1],C^\infty(K,E))\times [0,1]\to C^\infty(K,E)
\]
are $C^{\infty,0}$
and the evaluation map $\ve_1\colon C^\infty(K,E)\times K\to E$ is~$C^{\infty,\infty}$
(see \cite[Proposition 3.20]{AaS}).
Now the map
\[
C([0,1],C^\infty(K,E))\times C([0,1],K)\times [0,1]\times[0,1]\to E
\]
which sends $(\gamma,\eta,s,t)$ to $\ve_1(\ve_2(\gamma,s),\ve_3(\eta,t))$
is $C^{\infty,\infty,0,0}$ by \cite[Lemma~81]{Alz}
and thus $C^{\infty,0}$ in $((\gamma,\eta),(s,t))$ (using \cite[Lemma 77]{Alz} twice).
Composing with $t\mto (t,t)$, we deduce with \cite[Lemma~3.17]{AaS} that $\wh{h}$ is $C^{\infty,0}$.
\end{proof}
As before,
let $E:=\R^n$ and $K\sub E$ be a compact convex set with non-empty interior.
Let $\theta\in \; ]0,1/3]$. Since
\begin{equation}\label{thuslip}
q\colon C^\infty_{\partial K}(K,E)\to[0,\infty[,\quad \eta\mto \sup_{x\in K}\|\eta'(x)\|_{op}=\Lip(\eta)
\end{equation}
is a continuous seminorm, we deduce that
\[
\|.\|_{\infty,q}\colon C([0,1],C^\infty_{\partial K}(K,E))\to[0,\infty[,\quad
\gamma \mto\sup_{t\in [0,1]}q(\gamma(t))
\]
is a continuous seminorm on $C([0,1],C^\infty_{\partial K}(K,E))$. Hence
\[
Q:=\{\gamma \in C([0,1],C^\infty_{\partial K}(K,E))\colon \|\gamma\|_{\infty,q}<\theta\}
\]
is an open $0$-neighbourhood in $C([0,1],C^\infty_{\partial K}(K,E))$.
If $\gamma\in Q$, then
\[
\wh{\gamma}\colon [0,1]\times K\to E, \quad (t,x)\mto \gamma(t)(x)
\]
is a $C^{0,\infty}$-map (see \cite[Theorem 3.28\,(a)]{AaS}) which satisfies a global Lipschitz condition because
\begin{equation}\label{leqthe}
\sup_{t\in[0,1]}\Lip(\wh{\gamma}(t,\sbull))=\sup_{t\in[0,1]}\Lip(\gamma(t))\leq\theta<1.
\end{equation}
For $x\in K$, let
\[
d_{\partial K}(x):=\min_{y\in\partial K}\|y-x\|
\]
be the distance between $x$ and $\partial K$;
the minimum is attained as $\partial K$ is a non-empty compact set.
If $\eta\in C^\infty_{\partial K}(K,E)$ with $q(\eta)<\theta$
(e.g.\ $\eta=\gamma(t)$ for some $\gamma \in Q$ and $t\in[0,1]$),
then
\begin{equation}\label{estsizedist}
\|\eta(x)\|\leq \theta d_{\partial K}(x)\quad\mbox{for all $x\in K$.}
\end{equation}
To see this, given $x\in K$ pick $y\in \partial K$ such that $\|y-x\|=d_{\partial K}(x)$.
Noting that $\eta(y)=0$ since $\eta\in C^\infty_{\partial K}(K,E)$ and $y\in\partial K$, we deduce that
\[
\|\eta(x)\|=\|\eta(x)-\eta(y)\|\leq \Lip(\eta)\|y-x\|\leq\theta\|y-x\|=\theta d_{\partial K}(x).
\]
\begin{la}\label{smtoC0}
For each $\gamma\in Q$ and $x_0\in K$,
the initial value problem
\begin{equation}\label{inivalpr}
y'(t)=\wh{\gamma}(t,y(t)), \quad y(0)=x_0
\end{equation}
has a unique solution $y_{\gamma,x_0}\colon [0,1]\to K$.
The following map is smooth:
\[
\phi\colon Q\times K\to C([0,1],K),\quad (\gamma,x_0)\mto y_{\gamma,x_0}.
\]
\end{la}
\begin{proof}
Since $\wh{\gamma}$ satisfies a global Lipschitz condition, solutions to the initial value problem (\ref{inivalpr})
are unique whenever they exist. If $x_0\in\partial K$, then the constant function given by $y_{\gamma,x_0}(t):=x_0$
for all $t\in[0,1]$ solves (\ref{inivalpr}).
If $x_0\in K^0$, then
\[
R:=\frac{1}{2}d_{\partial K}(x_0)>0.
\]
Then $\wb{B}^E_R(x_0)\sub K^0$. For each $x\in \wb{B}^E_R(x_0)$, we have $d_{\partial K}(x)\leq d_{\partial K}(x_0)+R
=\frac{3}{2}d_{\partial K}(x_0)$ and hence
\[
\|\wh{\gamma}(t,x)\|=\|\gamma(t)(x)\|
\leq \theta d_{\partial K}(x)\leq \frac{3}{2}\theta d_{\partial K}(x_0)\leq \frac{1}{2}d_{\partial K}(x_0)=R,
\]
using (\ref{estsizedist}) and the hypothesis that $\theta\leq\frac{1}{3}$.
Therefore (\ref{inivalpr}) has a solution
\begin{equation}\label{wheresol}
y_{\gamma,x_0}\colon [0,1]\to \wb{B}^E_R(x_0)\sub K^0\;\;  \mbox{with $R=\frac{1}{2}d_{\partial K}(x_0)$,}
\end{equation}
by Lemma~\ref{quantlind}.
Now consider the map
$f\colon (Q\times K)\times C([0,1],K)\to C([0,1],E)$,
\[
f(\gamma,x_0,\eta)(t):=x_0+\int_0^t\wh{\gamma}(s,\eta(s))\, ds
\]
for $t\in[0,1]$. Using the continuous linear operator
\[
J\colon C([0,1],E)\to C([0,1],E),\;\; J(\zeta)(t):=\int_0^t\zeta(s)\,ds
\]
and the function
\[
h\colon Q\times C([0,1],K)\to C([0,1],K),\;\; h(\gamma,\eta)(t):=\wh{\gamma}(t,\eta(t))=\gamma(t)(\eta(t))
\]
which is smooth by Lemma~\ref{cmbine}, we have
\[
f(\gamma,x_0,\eta)=x_0+J(h(\gamma,\eta)),
\]
where also $K\to C([0,1],E)$, $x_0\mto x_0$ (the constant function $t\mto x_0$)
is smooth as it is the restriction of the continuous linear map $E\to C([0,1],E)$, $x_0\mto x_0$.
Hence $f$ is smooth. Moreover, $f$ defines a uniform family of contractions.
In fact, writing $f_{\gamma,x_0}:=f(\gamma,x_0,\sbull)\colon C([0,1],K)\to C([0,1],E)$
for $\gamma\in Q$ and $x_0\in K$, we have
\begin{eqnarray*}
\|f_{\gamma,x_0}(\eta_1)(t)-f_{\gamma,x_0}(\eta)(t)\|&\leq&
\int_0^t\|\gamma(s)(\eta_1(s))-\gamma(s)(\eta(s))\|\, ds\\
&\leq& \int_0^1\Lip(\gamma(s))\|\eta_1(s)-\eta(s)\|\,ds
\leq \theta\|\eta_1-\eta\|_\infty
\end{eqnarray*}
for all $\eta,\eta_1\in C([0,1],K)$ and $t\in[0,1]$ and hence
\[
\|f_{\gamma,x_0}(\eta_1)-f_{\gamma,x_0}(\eta)\|_\infty\leq\theta\|\eta_1-\eta\|_\infty
\]
with $\theta\leq\frac{1}{3}<1$ independent of $(\gamma,x_0)$.
By construction, $\phi(\gamma,x_0)=y_{\gamma,x_0}$ is the (unique) fixed point of
$f_{\gamma,x_0}$, for all $(\gamma,x_0)\in Q\times K$.
In particular, $y_{\gamma,x_0}\in C([0,1],K^0)$ (see (\ref{wheresol})) is the fixed point of $f_{\gamma,x_0}|_{C([0,1],K^0)}$
for $\gamma\in Q$ and $x_0\in K^0$.
Hence, applying Lemma~\ref{pardepfix}\,(a) to the restriction
\[
(Q\times K^0)\times C([0,1],K^0)\to C([0,1],E),\;\; (\gamma,x_0,\eta)\mto f(\gamma,x_0,\eta)
\]
of~$f$, which is a $C^\infty$-map on an open domain, we find that
$\phi|_{Q\times K^0}$
is smooth, and hence continuous. If $\gamma\in Q$ and $x_0\in \partial K$, we shall presently show that
\begin{equation}\label{lastbit}
\|\phi(\gamma_1,x_1)-\phi(\gamma,x_0)\|_\infty\leq\frac{3}{2}\|x_1-x_0\|
\end{equation}
for all $x_1\in K$ and $\gamma_1\in Q$.
Thus  $\phi$ will be continuous at $(\gamma,x_0)$
also in this case. Hence $\phi$ will be continuous
and so $\phi$ will be smooth, by Lemma~\ref{pardepfix}\,(b).
To establish (\ref{lastbit}), assume $x_1\in\partial K$ first.
Then $\phi(\gamma_1,x_1)(t)=x_1$ for all $t$ and since also $\phi(\gamma,x_0)(t)=x_0$ for all~$t$,
we obtain
\[
\|\phi(\gamma_1,x_1)-\phi(\gamma,x_0)\|_\infty=\|x_1-x_0\|\leq\frac{3}{2}\|x_1-x_0\|.
\]
It remains to consider the case $x_1\in K^0$. Then $\phi(\gamma_1,x_1)(t)\in \wb{B}^E_R(x_1)$
for all $t\in[0,1]$ with $R:=\frac{1}{2}d_{\partial K}(x_1)\leq\frac{1}{2}\|x_1-x_0\|$
(see (\ref{wheresol})).
Hence
\begin{eqnarray*}
\|\phi(\gamma_1,x_1)(t)-\phi(\gamma,x_0)(t)\|&=& \|\phi(\gamma_1,x_1)(t)-x_0\|\\
&\leq& \|\phi(\gamma_1,x_1)(t)-x_1\|+\|x_1-x_0\|\\
&\leq & R+\|x_1-x_0\|\leq\frac{3}{2}\|x_1-x_0\|
\end{eqnarray*}
also in this case.
\end{proof}
\begin{la}\label{smtoC1}
In the situation of Lemma~\emph{\ref{smtoC0}}, also $\phi\colon Q\times K\to C^1([0,1],E)$, $(\gamma,x_0)\mto y_{\gamma,x_0}$
is smooth.
\end{la}
\begin{proof}
The map $\Lambda\colon C^1([0,1],E)\to C([0,1],E)\times C([0,1],E)$, $\eta\mto(\eta,\eta')$ is
a linear topological embedding with closed image \cite[Lemma~2.7]{AaS}.
It therefore suffices to show that both components of $\Lambda\circ\phi$ are smooth
(see \cite{GaN}; cf.\ \cite[Lemma 10.1 and 10.2]{BGN}).
The first of these is $\phi$ as a map to $C([0,1],E)$ and hence smooth, by Lemma~\ref{smtoC0}.
The second component is the map
\[
Q\times K\to C([0,1],E),\quad (\gamma,x_0)\mto (y_{\gamma,x_0})',\quad\mbox{where}
\]
\[
(\forall t\in[0,1])\;\;\;
(y_{\gamma,x_0})'(t)=\wh{\gamma}(t,y_{\gamma,x_0}(t))=\wh{\gamma}(t,\phi(\gamma,x_0)(t))
\]
and thus $(y_{\gamma,x_0})'=h(\gamma,\phi(\gamma,x_0))$
with $h$ as  Lemma~\ref{cmbine}, which is a smooth $C([0,1],E)$-valued function of $(\gamma,x_0)\in Q\times K$
by smoothness of $h$ and smoothness of $\phi$ as a map to $C([0,1],K)$.
\end{proof}
Recall that $G:=\Diff_{\partial K}(K)$ is an open subset of the closed affine subspace
$\id_K+C^\infty_{\partial K}(K,E)$ of $C^\infty(K,E)$
and hence a smooth submanifold of $C^\infty(K,E)$ modelled
on $C^\infty_{\partial K}(K,E)$. Identifying the tangent bundle of the locally convex space
$C^\infty(K,E)$ with
$C^\infty(K,E)\times C^\infty(K,E)$ in the usual way, we also obtain an identification
\[
T\Diff_{\partial K}(K)=\Diff_{\partial K}(K)\times C^\infty_{\partial K}(K,E).
\]
This enables $L(G)=T_{\id_K}(\Diff_{\partial K}(K))=\{\id_K\}\times C^\infty_{\partial K}(K,E)$ to be identified with $C^\infty_{\partial K}(K,E)$ (forgetting the first component).
\begin{la}\label{spotevol}
If $\gamma\colon [0,1]\to L(G)=C^\infty_{\partial K}(K,E)$ is a continuous curve
and $\eta\colon [0,1]\to G$ a $C^1$-curve\footnote{Since $G$ is a submanifold, this is simply a $C^1$-curve to $C^\infty(K,E)$ with image in~$G$.} with $\eta(0)=\id_K$, then $\eta=\Evol^r(\gamma)$ if and only if
$\eta$ is the flow of the time-dependent vector field~$\gamma$
for initial time $t_0=0$, i.e.,
\[
\frac{\partial}{\partial t}(\eta(t)(x))=\gamma(t)(\eta(t)(x))\quad\mbox{for all $t\in[0,1]$}
\]
and $\eta(0)(x)=x$.
\end{la}
\begin{proof}
For each $\psi\in \Diff_{\partial K}(K)$, the right translation
\[
r_\psi\colon C^\infty(K,E)\to C^\infty(K,E),\quad \zeta\mto\zeta\circ\psi
\]
is a continuous linear map (by smoothness of $\Gamma$ in Lemma~\ref{premultHamz}), entailing that
\[
T(r_\psi)(\zeta,\theta)=(r_\psi(\zeta),r_\psi(\theta))
\]
for all $\zeta,\theta\in C^\infty(K,E)$. As a consequence, the tangent map of the restriction
\[
\rho_\psi\colon \Diff_{\partial K}(K)\to\Diff_{\partial K}(K),\quad \xi \mto \xi \circ \psi
\]
to the submanifold $\Diff_{\partial K}(K)$ is given by
\[
T\rho_\psi(\xi,\theta)=(\rho_\psi(\xi),r_\psi(\theta))=(\xi\circ \psi,\theta\circ \psi)
\]
for all $\xi\in \Diff_{\partial K}(K)$ and $\theta\in C^\infty_{\partial K}(K,E)$.
Hence, if $\gamma$ and $\eta$ are as described in the lemma,
then
$\eta=\Evol^r(\gamma)$ if and only if
\begin{equation}\label{tgether}
\eta'(t)=\gamma(t)\circ \eta(t)
\end{equation}
for all $t\in [0,1]$, where $\eta'(t)\in C^\infty_{\partial K}(K,E)\sub C^\infty(K,E)$
is the derivative of $\eta$ as a map to the locally convex space $C^\infty(K,E)$.
Applying the continuous linear point evaluations $\ve_x\colon C^\infty(K,E)\to E$, $\zeta\mto\zeta(x)$
for $x\in K$, which separate points on $C^\infty(K,E)$,
we see that (\ref{tgether}) is equivalent to
\[
\frac{\partial}{\partial t} (\eta(t)(x))=\gamma(t)(\eta(t)(x))
\]
for all $x\in K$, i.e., $t\mto \eta(t)(x)$ is the solution $y_{\gamma,x}$ of the initial value
problem $y'(t)=\gamma(t)(y(t))$, $y(0)=x$.
\end{proof}
\begin{prop}
The Lie group
$\Diff_{\partial K}(K)$ is $C^0$-regular,
for each compact convex subset $K\sub \R^n$ with non-empty interior.
\end{prop}
\begin{proof}
Let $E:=\R^n$. Let $Q$ and $\phi \colon Q\times K\to C^1([0,1],E)$
be as in Lemma~\ref{smtoC1}. Since $\phi$ is smooth,
\[
\wh{\phi}\colon (Q\times K)\times [0,1]\to E,\quad (\gamma,x,t)\mto\phi(\gamma,x)(t)
\]
is $C^{\infty,1}$ (see \cite[Theorem 3.28\,(a)]{AaS}) and thus $C^{\infty,\infty,1}$ as a map on the threefold product
$Q\times K\times [0,1]$ (see \cite[Lemma 81]{Alz}).
Therefore
\[
\psi\colon Q\times [0,1]\times K\to E,\quad \psi(\gamma,t,x):=\wh{\phi}(\gamma,x,t)
\]
is $C^{\infty,1,\infty}$ and thus
\[
\psi^\vee\colon Q\times [0,1]\to C^\infty(K,E),\quad \psi^\vee(\gamma,t)(x):=\psi(\gamma,t,x)
\]
is $C^{\infty,1}$ (see \cite[Theorem 94]{Alz}).
We have $\psi^\vee(\gamma,t)(x)=\phi(\gamma,x)(t)=y_{\gamma,x}(t)$ for all $\gamma\in Q$, $x\in K$ and $t\in [0,1]$,
whence $\psi^\vee(\gamma,t)(x)=x$ whenever $x\in\partial K$.
We can therefore consider $\psi^\vee$ as a continuous
map to the affine vector subspace $\id_K+C^\infty_{\partial K}(K,E)$, in which $\Diff_{\partial K}(K)$
is a neighbourhood of $\id_K$.
Since $\psi^\vee(0,t)(x)=y_{0,x}(t)=x$, we have $\psi^\vee(0,t)=\id_K$ for all $t\in [0,1]$.
Using the Wallace Lemma \cite[5.12]{Kel}, we find an open $0$-neighbourhood $P\sub Q$ such that
\[
\psi^\vee(P\times [0,1])\sub \Diff_{\partial K}(K).
\] 
For fixed $\gamma\in P$, the map $\psi^\vee(\gamma,\sbull)\colon [0,1]\to \Diff_{\partial K}(K)$
is $C^1$ and $\psi^\vee(\gamma,t)(x)=y_{\gamma,x}(t)$ for all $x\in K$ and $t\in [0,1]$, whence
\[
\psi^\vee(\gamma,\sbull)=\Evol^r(\gamma)
\]
by Lemma~\ref{spotevol}. Thus $\Evol^r\colon P\to C^1([0,1],G)$ exists.
If we can show that
\begin{equation}\label{themap}
\evol^r\colon P \to G,\quad \evol^r(\gamma):=\Evol^r(\gamma)(1)
\end{equation}
is smooth, then $G$ will be $C^0$-regular by \cite[Lemma~9.5]{NaS}.
Since $G=\Diff_{\partial K}(K)$ is a smooth submanifold of $C^\infty(K,E)$,
the map $\evol^r$ from (\ref{themap})
will be smooth as a map to $G$ if we can show that it is smooth as a map to $C^\infty(K,E)$.
Now, because
$\psi^\vee$ is $C^{\infty,1}$, the map
\[
g:=(\psi^\vee)^\vee\colon Q\to C^1([0,1],C^\infty(K,E)),\;\;
g(\gamma)(t):=(\psi^\vee)(\gamma,t)
\]
is $C^\infty$ by \cite[Theorem 3.25\,(a)]{AaS}. Using that
the point evaluation\linebreak
$\ve_1\colon C^1([0,1],C^\infty(K,E))\to C^\infty(K,E)$, $\zeta\mto\zeta(1)$
is continuous linear and hence smooth, we deduce that also
\[
\ve_1\circ g\colon Q\to C^\infty(K,E),\quad \gamma\mto g(\gamma)(1)
\]
is smooth. But $g(\gamma)(1)(x)=\psi^\vee(\gamma,1)(x)=\Evol^r(\gamma)(1)(x)=\evol^r(\gamma)(x)$
for all $x\in K$ and thus $g(\gamma)(1)=\evol^r(\gamma)$ for all $\gamma\in P$.
Thus $\evol^r|_P$ is $C^\infty$.
\end{proof}
\section{Consequences for initial value problems on compact convex sets}\label{secODE}
We start with some preparatory considerations.\\[2.5mm]
Let $E:=\R^n$, $K\sub E$ be a compact convex set with non-empty interior,
$J\sub \R$ be a non-degenerate interval, $t_0\in J$ and
$f\colon J\times K\to E$ be a $C^{0,\infty}$-map
such that $f(t,x)=0$ for all $t\in J$ and $x\in\partial K$.
\begin{numba}\label{thusuni}
For each compact subinterval $C\sub J$, we have that
\[
\sup_{t\in C}\Lip(f(t,\sbull))=\sup_{t\in C}q(f^\vee(t))<\infty,
\]
with $q$ as in (\ref{thuslip}).
Hence $f$ satisfies a local Lipschitz condition and hence
solutions to $y'(t)=f(t,y(t))$, $y(t_0)=x_0$ are unique on their
interval of definition
(if they exist), for all $t_0\in J$ and $x_0\in K$.
As a consequence, there is a unique maximal solution $y_{t_0,x_0}\colon J_{t_0,x_0}\to K$
to the preceding initial value problem, such that all other solutions are restrictions
of $y_{t_0,x_0}$ to subintervals of~$J_{t_0,y_0}$.
\end{numba}
\begin{numba}\label{thusrewr}
For fixed $t_0,t\in J$, the map
\[
g\colon [0,1]\times K\to E,\quad (\tau,x)\mto (t-t_0)f(t_0+\tau (t-t_0),x)
\]
is $C^{0,\infty}$ and hence $g^\vee\colon [0,1]\to C^\infty_{\partial K}(K,E)\sub C^\infty(K,E)$
is continuous by \cite[Theorem 3.25\,(a)]{AaS}. Thus $\Evol^r(g^\vee)\colon [0,1]\to \Diff_{\partial K}(K)$
is $C^1$ and we know that, for $x_0\in K$,
\[
[0,1]\to K, \;\; \tau\mto \Evol^r(g^\vee)(\tau)(x_0)
\]
is the solution to $y'(\tau)=g(\tau,y(\tau))$, $y(0)=x_0$.
If $t-t_0\not=0$, using the Chain Rule this implies that
\begin{equation}\label{ugly}
s\mto \Evol^r(g^\vee)\left(\frac{s-t_0}{t-t_0}\right)(x_0)
\end{equation}
(for $s$ in the interval $I$ between $t_0$ and $t$) solves $y'(s)=f(s,y(s))$, $y(t_0)=x_0$.
Thus $I$ is contained in the domain of definition $J_{t_0,x_0}$ of the maximal solution
$y_{t_0,x_0}$ of the latter initial value problem, and $y_{t_0,x_0}(s)$
is given by (\ref{ugly}) for all $s\in I$. As $t$ was arbitrary, we deduce that
$y_{t_0,x_0}$ is defined on all of~$J$.
For later use, let us take $s:=t$ in (\ref{ugly}); we obtain
\begin{equation}\label{ugly2}
y_{t_0,x_0}(t)=\Evol^r(g^\vee)(1)(x_0)=\evol^r(g^\vee)(x_0).
\end{equation}
\end{numba}
\begin{prop}
Let $J\sub\R$ be a non-degenerate interval, $K\sub\R^n$ be a compact
convex set with non-empty interior and $P\sub F$ be a locally convex subset with dense interior
in a locally convex space~$F$. Let $r,s\in\N_0\cup\{\infty\}$ and
\[
f\colon P\times J\times K\to\R^n
\]
be a $C^{r,s,\infty}$-map such that $f(p,t,x)=0$ for all $p\in P$, $t\in J$ and $x\in\partial K$.
Then the initial value problem
\[
y'(t)=f(p,t,y(t)),\quad y(t_0)=x_0
\]
has a unique solution $y_{p,t_0,x_0}\colon J\to K$ defined on all of~$J$,
for all $p\in P$, $t_0\in J$, and $x_0\in K$.
The associated flow
\[
\Phi\colon P\times (J\times J)\times K \to K,\quad (p,t_0,t,x_0)\mto y_{p,t_0,x_0}(t)
\]
is $C^{r,s,\infty}$.
\end{prop}
\begin{proof}
The uniqueness assertion was settled in \ref{thusuni} and existence in \ref{thusrewr}.
We now use that the mapping $h\colon P\times (J\times J) \times [0,1]\times K\to \R^n$,
\[
h(p,t_0,t,\tau ,x):=(t-t_0)f(p,t_0+\tau (t-t_0),x)
\]
is $C^{r,s,0,\infty}$, whence
\[
h^\vee\colon P\times (J\times J)\times [0,1]\to C^\infty(K,\R^n),\quad h^\vee(p,t_0,t,\tau)(x):=h(p,t_0,t,\tau,x)
\]
is $C^{r,s,0}$, by \cite[Theorem 94]{Alz}. Then $h^\vee$ is also $C^{r,s,0}$ as a map to
the closed vector subspace $C^\infty_{\partial K}(K,\R^n)$ of $C^\infty(K,\R^n)$.
Using \cite[Theorem 94]{Alz} again, we see that
\[
(h^\vee)^\vee\colon P\times (J\times J)\to C([0,1],C^\infty_{\partial K}(K,\R^n)),\quad
(h^\vee)^\vee(p,t_0,t)(\tau):=h^\vee(p,t_0,t,\tau)
\]
is $C^{r,s}$.
Hence also
\[
g:=\evol^r\circ (h^\vee)^\vee\colon P\times (J\times J)\to\Diff_{\partial K}(K)\sub C^\infty(K,\R^n)
\]
is $C^{r,s}$, where $\evol^r\colon C([0,1],C^\infty_{\partial K}(K,\R^n))\to\Diff_{\partial K}(K)$
is the right evolution map which is smooth by $C^0$-regularity of $\Diff_{\partial K}(K)$.
Hence
\[
\wh{g}\colon P\times (J\times J)\times K\to K,\quad \wh{g}(p,t_0,t,x_0):=g(p,t_0,t)(x_0)
\]
is $C^{r,s,\infty}$ (see \cite[Theorem 96]{Alz}). Since $\Phi=\wh{g}$ (cf.\ (\ref{ugly2})),
the proposition is established.
\end{proof}
Theorem B is a special case:
If $f$ is $C^\infty$ and thus $C^{\infty,\infty,\infty}$ in the preceding proposition,
then $\Phi$ is $C^{\infty,\infty,\infty}$ and hence $C^\infty$
(by \cite[Remark 79]{Alz}).
\section{The Lie group {\boldmath$\Diff_{\flt}(K)$}}\label{secflat}
Let $K\sub E=\R^n$ as before.
For $k\in\N$, let $P_k(E)$ be the finite-dimensional vector space of all homogeneous polynomials $p\colon E\to E$ of order~$k$
and
\[
P_{\leq k}(E)_0\cong\bigoplus_{j=1}^k P_j(E)
\]
be the finite-dimensional vector space of all polynomial functions
$p\colon E\to E$ of degree $\leq k$ such that $p(0)=0$.
Then $P_1(E)=\cL(E)$ is the space of linear endomorphisms of~$E$.
Given $p,q\in P_{\leq k}(E)_0$, let $p\diamond_k q$ be the $k$th order Taylor polynomial of $p\circ q$.
Thus, if $p\circ q=\sum_{j=1}^{k^2} h_j$ with homogeneous polynomials $h_j\colon E\to E$
of order~$j$, then $p\diamond_k q$ is given by the truncated composition
\[
p\diamond_k q=\sum_{j=1}^k h_j.
\]
It is clear that the map $\diamond_k$ is smooth and that $(P_{\leq k}(E)_0,\diamond_k)$ is a monoid
with $\id_E$ as the neutral element and open unit group
\[
P_{\leq k}(E)_0^\times = \GL(E)\times\bigoplus_{j=2}^k P_j(E).
\]
Also the inversion map is smooth as it takes $p$ to the $k$th order Taylor polynomial of $p|_U^{-1}$ at~$0$
(for some open $0$-neighbourhood $U\sub E$), which depends smoothly on~$p$ (e.g., by the Inverse
Function Theorem with Parameters in~\cite{IMP}).
Thus $P_{\leq k}(E)_0^\times$ is a (finite-dimensional) smooth Lie group.
Note that each $x_0\in\partial K$ is a fixed point for each $\phi\in \Diff_{\partial K}(K)$,
and the map
\[
f_{x_0,k}\colon \Diff_{\partial K}(K)\to P_{\leq k}(E)_0^\times
\]
which takes $\phi$ to the $k$th order Taylor polynomial of $\phi- x_0$ around $x_0$
is a smooth group homomorphism, for each $k\in\N$.
Let $\cO\colon \partial K\to \N_0\cup\{\infty\}$ be a function
and
\[
\Diff_\cO(K)
\]
be the group of all $\phi\in\Diff_{\partial K}(K)$ such that $f_{x_0,k}(\phi)=\id_E$
for all $x_0\in\partial K$ and all $k\in\N$ such that $k\leq \cO(x_0)$
(if $\cO(x_0)=0$, then the condition is vacuous).
Then $\Diff_\cO(K)$
is a closed normal subgroup of $\Diff_{\partial K}(K)$.
Taking $\cO(x_0):=\infty$ for each $x_0\in\partial K$, we obtain
\[
\Diff_{\flt}(K):=\Diff_\cO(K)
\]
as a special case.
Returning to general $\cO$, it is clear that
\[
C^\infty_\cO(K,E):=
\{\eta\in C^\infty_{\partial K}(K,E)\colon (\forall x_0\in\partial K)(\forall \N_0\ni k\leq \cO(x_0))\;\; \eta^{(k)}(x_0)=0\}
\]
is a closed vector subspace of $C^\infty_{\partial K}(K,E)$.
\begin{prop}
For each compact convex subset $K\sub\R^n=:E$ with non-empty interior
and each $\cO\colon\partial K\to\N_0\cup\{\infty\}$, the subgroup
$\Diff_\cO(K)$ is a submanifold of $\Diff_{\partial K}(K)$ modelled on $C^\infty_\cO(K,E)$
and hence a Lie group. The Lie group $\Diff_\cO(K)$ is $C^0$-regular.
\end{prop}
\begin{proof}
The chart $\Phi\colon \Diff_{\partial K}(K)\to\Omega\sub C^\infty_{\partial K}(K,E)$, $\phi\mto\phi-\id_K$ takes
$\Diff_\cO(K)$ onto $\Omega\cap C^\infty_\cO(K,E)$, whence $\Diff_\cO(K)$ is a submanifold
modelled on the closed vector subspace $C^\infty_\cO(K,E)$ of $C^\infty_{\partial K}(K,E)$.
Let $g_{x_0,k}\colon \Diff_{\partial K}(K)\to P_{\leq k}(E)_0^\times$ be the trivial homomorphism
taking each $\phi$ to the neutral element~$\id_E$.
Then the Lie subgroup $\Diff_\cO(K)$ coincides with the equalizer
\[
\{\phi\in\Diff_{\partial K}(K)\colon (\forall x_0\in\partial K)(\forall \N\ni k\leq\cO(x_0))\;\;
f_{x_0,k}(\phi)=g_{x_0,k}(\phi)\}
\]
of the given pairs $(f_{x_0,k},g_{x_0,k})$ of smooth homomorphisms of Lie groups.
Hence $\Diff_\cO(K)$ inherits the $C^0$-regularity from the ambient $C^0$-regular Lie group
$\Diff_{\partial K}(K)$, by \cite[Theorem G]{SEM}.
\end{proof}
\appendix
\section{Proof of a folklore lemma}\label{appA}
{\bf Proof of Lemma~\ref{clogra}.}
(a) If $(x_\alpha,y_\alpha)_{\alpha\in A}$ is a net in $\graph(f)$ which converges to some
$(x,y)\in X\times K$, then the net of the $y_\alpha=f(x_\alpha)$ converges to $f(x)$,
by continuity of~$f$. Since also $y_\alpha\to y$ and limits in Hausdorff spaces are unique,
we obtain $(x,y)=(x,f(x))\in\graph(f)$.

(b) We show that if $f$ is not continuous, then $\graph(f)$ is not closed. 
Now, if $f$ is not continuous, then $f$ fails to be continuous at some $x\in X$.
Hence, there is an open neighbourhood $V\sub K$ of $f(x)$
such that $f^{-1}(V)$ is not a neighbourhood of~$x$ in~$X$.
Thus $U\setminus f^{-1}(V)\not=\emptyset$ for each neighbourhood~$U$ of~$x$ in~$X$.
Pick $x_U\in U\setminus f^{-1}(V)$. Since $f(x_U)\in K\setminus V$ and $K\setminus V$ is compact,
there is a convergent subnet $(f(x_{U(\alpha)}))_{\alpha\in A}$ (indexed by some directed set $(A,\leq)$).
Let $y\in K\setminus V$ be its limit.
Then $(x_{U(\alpha)},f(x_{U(\alpha)}))_{\alpha\in A}$ is a net in $\graph(f)$
which converges to $(x,y)$. We have $y\not= f(x)$ (since $y\in K\setminus V$ but $f(x)\in V$)
and hence $(x,y)\not\in\graph(f)$. Thus $\graph(f)$ is not closed.\;\Punkt\vspace{-2mm}
{\small
{\bf Helge  Gl\"{o}ckner}, Institut f\"{u}r Mathematik, Universit\"at Paderborn,\\
Warburger Str.\ 100, 33098 Paderborn, Germany; {\tt  glockner@math.upb.de}\\[2mm]
{\bf Karl-Hermann Neeb}, Department Mathematik, FAU Erlangen-N\"{u}rnberg,\\
Cauerstr.\ 11, 91058 Erlangen,
Germany; {\tt neeb@math.fau.de}}\vfill
\end{document}